\documentclass[11pt]{amsart}

\usepackage{amsmath}
\usepackage{amssymb}
\usepackage{bbm}
\usepackage{esint} 
\usepackage[colorlinks,citecolor=red,pagebackref,hypertexnames=false]{hyperref}
\usepackage{esint} 
\usepackage{enumitem} 
\usepackage{verbatim} 

\newcommand{\C}{{\mathbb C}}       
\newcommand{\R}{{\mathbb R}}       

\newcommand{\HH}{{\mathcal H}}

\newcommand{\EE}{{\mathcal E}}

\newcommand{\diam}{{\rm diam}}
\newcommand{\dist}{{\rm dist}}

\newcommand{\fiproof}{{\hspace*{\fill} $\square$ \vspace{2pt}}}

\newcommand{\rf}[1]{{(\ref{#1})}}

\newcommand{\supp}{\operatorname{supp}}

\newcommand{\ve}{{\varepsilon}}
\newcommand{\vv}{{\vspace{2mm}}}
\newcommand{\vvv}{{\vspace{3mm}}}
\newcommand{\wt}[1]{{\widetilde{#1}}}

\newcommand{\rad}{{\operatorname{rad}}}

\newcommand{\capp}{\operatorname{Cap}}

\def\XXint#1#2#3{{\setbox0=\hbox{$#1{#2#3}{\int}$ }
\vcenter{\hbox{$#2#3$ }}\kern-.58\wd0}}

\newtheorem{theorem}{Theorem}[section]
\newtheorem{lemma}[theorem]{Lemma}
\newtheorem{mlemma}[theorem]{Main Lemma}

\newtheorem*{theorem*}{Theorem}

\theoremstyle{definition}

\theoremstyle{remark}
\newtheorem{rem}[theorem]{\bf Remark}

\numberwithin{equation}{section}

\newcommand{\brem}{\begin{rem}}
\newcommand{\erem}{\end{rem}}

\usepackage{xcolor}

\begin{document}

\title[Dimension of harmonic measure on some flat  sets of fractional dimension]
{The dimension of harmonic measure on some AD-regular flat sets of fractional dimension}

\begin{abstract}
In this paper it is shown that if $E\subset\R^{n+1}$ is an $s$-AD regular compact set, with $s\in [n-\frac12,n)$, and $E$ is contained in a hyperplane or, more generally, in an $n$-dimensional $C^1$ manifold, then the Hausdorff dimension of the harmonic measure for the domain $\R^{n+1}\setminus E$ is strictly smaller than $s$, i.e., than
the Hausdorff dimension of~$E$.

\end{abstract}

\author{Xavier Tolsa}

\address{ICREA, Barcelona\\
Dept. de Matem\`atiques, Universitat Aut\`onoma de Barcelona \\
and Centre de Recerca Matem\`atica, Barcelona, Catalonia.}
\email{xtolsa@mat.uab.cat}

\thanks{The author is supported by the European Research Council (ERC) under the European Union's Horizon 2020 research and innovation programme (grant agreement 101018680). Also partially supported by MICINN (Spain) under the grant PID2020-114167GB-I00, the María de Maeztu Program for units of excellence (Spain) (CEX2020-001084-M), and 2021-SGR-00071 (Catalonia).
}

\maketitle

\section{Introduction}

The study of the metric and geometric properties of harmonic measure is a classical topic in analysis. At least, this goes back  to the
work of the Riesz brothers \cite{RR} about the mutual absolute continuity between harmonic measure and arc-length measure on Jordan domains with rectifiable boundaries.
In the last years, the application of new ideas and techniques originating from harmonic analysis, PDE's, and geometric measure theory 
has allowed remarkable advances, especially when the boundary of the domain has codimension one. See for example 
\cite{AHM3TV}, \cite{AHMMT}, \cite{Azzam-semi}, \cite{GMT}, \cite{HMM}.
The case of codimension different from one is less studied and presents more difficulties, due to the fact that notions such as rectifiability or $L^2$ boundedness of Riesz transforms seem to play no role. In this paper we will focus on the
behavior of harmonic measure on AD-regular boundaries of codimension larger than one.

Recall that, given a domain $\Omega\subset\R^{n+1}$ and $p\in\Omega$, the harmonic measure $\omega^p$ for $\Omega$ with pole in $p$ is the
Borel measure supported in $\partial\Omega$ such that, for any $f\in C_c(\partial\Omega)$ which extends continuously to the whole $\overline \Omega$ and is harmonic in $\Omega$ (and vanishes at $\infty$ in case that $n>1$ and $\Omega$ is unbounded), we have
$$f(p) = \int_{\partial\Omega} f\,d\omega^p.$$
One of the most important problems about harmonic measure consists in estimating its dimension. Recall that, for any Borel measure $\nu$
in $\R^{n+1}$, its (Hausdorff) dimension, denoted by $\dim\nu$, is defined by
$$\dim\nu = \inf\,\{\dim F: F\subset \R^{n+1} \text{ Borel, } \nu(F^c)=0\},$$
where $\dim F$ stands for the Hausdorff dimension of $F$.
Notice that $\dim\omega^p$ does not depend on the precise pole $p$, assuming $\Omega$ to be connected. So quite often we will write
$\dim\omega$ instead of $\dim\omega^p$.

A very relevant result about the dimension of harmonic measure was obtained by Makarov \cite{Makarov} in 1985, when he showed
that for any simply connected domain in the plane, $\dim\omega=1$ (in spite of the fact that the boundary of the domain may have dimension larger than $1$). Later on, Jones and Wolff \cite{JW} showed that, for any arbitrary domain in the plane, $\dim\omega\leq 1$.
Subsequently, Wolff \cite{Wolff1} sharpened this result by proving that $\omega$ must be concentrated on a set of $\sigma$-finite length.
In the higher dimensional case $n>1$, the situation is more complicated. 
On the one hand, Bourgain \cite{Bourgain} proved in 1987 that there exists some constant $\ve_n>0$ just depending on $n$ such that
$\dim\omega\leq n+1-\ve_n$ for any $\Omega\subset\R^{n+1}$.
A natural guess would be that one could take $\ve_n=1$, so that
$\dim\omega\leq n$ for any domain of $\R^{n+1}$, analogously to what happens in the plane. However, this was disproved by Wolff in his celebrated work \cite{Wolff2}, where
he managed to construct a snowflake type domain $\Omega\subset\R^{n+1}$ satisfying $\dim\omega>n$.
A difficult open question in the area consists in finding the optimal value of the constant $\ve_n$ such that $\dim\omega\leq n+1-\ve_n$ for any $\Omega\subset\R^{n+1}$. See for example \cite{Jones-scaling}.

When the (Hausdorff) codimension of $\partial\Omega$ is not $1$ or $\partial\Omega$ is of fractal type, many examples show that we may have $\dim\omega<\dim\partial\Omega$. This is the so-called ``dimension drop'' for harmonic measure, which seems to be a frequent phenomenon. This was first
observed by Carleson \cite{Carleson} for some domains defined as complements of suitable Cantor type sets in the plane. 
Later on, Jones and Wolff showed a similar result for some boundaries $\partial\Omega$ satisfying some uniform disconnectedness property in the plane (see  \cite[Section X.I.2]{GM}). 
In \cite{Volberg1} and \cite{Volberg2} (see also \cite{MV}) Volberg proved the dimension drop of harmonic measure for a large class of Cantor repellers in the plane (see \cite{Volberg2}
for the notion of Cantor repeller and the precise statement of the result). Later on, Batakis \cite{Batakis} proved analogous results for the harmonic
measure for a large class of (complements of) self-similar sets in the plane. In \cite{UZ}
Urba\'nski and Zdunik showed that the dimension drop also occurs for the attractors of conformal iterated function systems (IFS) when either the limit set is contained in a real-analytic curve, if the IFS consists of similarities only, or if the IFS is irregular.
Another related result was obtained more recently by Batakis and Zdunik in \cite{BZ} for another class of IFS.

In view of the results described above, it is natural to wonder if the dimension drop for harmonic measure
occurs for more general, ``not dynamically generated'', subsets of $\R^{n+1}$ with boundaries with fractional dimension, like 
domains with AD-regular boundaries of fractional dimension. See Conjecture 1.9 from \cite{Volberg3}.
Recall that, given $s>0$ and $C_0>1$, a set
$E\subset\R^{n+1}$ is called AD-regular (or $s$-AD regular, or $(s,C_0)$-AD regular, if we want to be more precise) if
$$C_0^{-1}r^s\leq \HH^s(E\cap B(x,r)) \leq C_0\,r^s\quad \mbox{ for all $x\in E$ and $0<r\leq \diam(E)$.}$$
An answer in the affirmative to the above question was given by Azzam \cite{Azzam-drop} in the case $s\in (n,n+1)$ (i.e., for codimension smaller than $1$). The case of codimension larger than $1$ (with\footnote{When $s\leq n-1$ and $\partial\Omega$ is $s$-AD regular, this is a polar set and harmonic measure on $\partial\Omega$ is not defined.}
$s\in (n-1,n)$) is more challenging. In fact, at the time of writing this paper, Guy David, Cole Jeznach, and Antoine Julia \cite{DJJ} have informed me that, for some $s\in (0,1)$, they have
managed to construct an $s$-AD-regular set $E\subset \R^2$,  such that the harmonic measure for $\Omega=\R^2\setminus E$
is also $s$-AD-regular, and thus mutually absolutely continuous with $\HH^s|_E$. A previous (unpublished) example 
of Chris Bishop also showed that harmonic measure can be mutually absolutely continuous with $H^s$ for $s<1$ for some domains in the plane. However in Bishops's example $\partial\Omega$ is not $s$-AD regular. See also \cite{Batakis} for a related example consisting of a 
non-AD-regular Cantor set whose dimension coincides with the dimension of harmonic measure.

The main result of the present paper goes in the converse direction, and it shows that the dimension drop occurs for
$s$-AD regular subsets of hyperplanes in a suitable range of values of $s$ of codimension larger than $1$.
In the plane, for example, this occurs for $s$-AD regular sets contained in a line, for $s\in[\frac12,1)$. The precise result is the following.

\begin{theorem}\label{teo1}
For $n\geq1$ and $s\in [n-\frac12,n)$, let $E\subset \R^{n+1}$ be an $s$-AD regular compact set contained in a hyperplane. Let $\Omega=\R^{n+1}\setminus E$ and denote
by $\omega$ the harmonic measure for $\Omega$. Then
$\dim \omega < s.$
\end{theorem}
\vv

One can quantify the result above and show that $\dim \omega \leq s -\kappa$, with 
$\kappa = \kappa(n,s,C_0)>0$, for $E\subset\R^{n+1}$ being
$(s,C_0)$-AD regular. Further, the assumption that $E$ is contained in a hyperplane can be replaced by $E$ being contained in a $C^1$ n-dimensional manifold in $\R^{n+1}$ (see Theorem \ref{teoC1} below).

I don't know if the threshold $n-\frac12$ is sharp in Theorem \ref{teo1}. 
However, 
we can go a bit below the threshold $n-\frac12$ in the following sense:

\begin{theorem}\label{teo2}
For $s>0$, $C_0>1$, let $E\subset \R^{n+1}$ be an $(s,C_0)$-AD regular compact set contained in a hyperplane, and let $\Omega=
\R^{n+1}\setminus E$. 
 Then there exists $\ve>0$ small enough depending on $n$ and $C_0$ such that if $s\in [n-\frac12-\ve,n)$, we have $\dim \omega<s$.
 \end{theorem}

On the other hand, Theorem \ref{teo1} does not hold for $s$ close enough to $n-1$. Indeed, some of the aforementioned examples of the authors in \cite{DJJ} are $s$-AD regular subsets of
the real line in the plane (with $s<1/2$, of course). So an interesting open problem consists in finding the sharp threshold $s_0$ such that for all $s$-AD regular sets
with $s\in (s_0,1)$ contained in a line the dimension drop for harmonic measure occurs.
Also, for $s$-AD regular sets not contained in a line, it is an open question if a similar threshold exists.

The proofs of Theorem \ref{teo1} and Theorem \ref{teo2} rely on an idea originating from Bourgain \cite{Bourgain} (see Lemma \ref{lem-dim}) and they follow
an approach similar to the one of \cite{Batakis} and
\cite{Azzam-drop}. Using touching point arguments, the maximum
principle, and suitable modifications of the domain, we are led to  some estimates involving the harmonic measure for the complement of a segment in the plane and for the complement of a suitable flat annulus for $n>1$. In the plane, such harmonic measure can be computed explicitly by means of a conformal transformation while in higher dimensions we need more elaborated estimates which also rely on the conformal
transformation from the planar case. The arguments make an essential use of the fact that the boundary $\partial\Omega$ is contained in
a hyperplane and so they cannot be extended to arbitrary $s$-AD regular sets.
\vv


\section{Preliminaries}

In the paper, constants denoted by $C$ or $c$ depend just on the dimension and perhaps other fixed
parameters, such as the parameter $s$ in Theorem \ref{teo1}, for example. We will write $a\lesssim b$ if there is $C>0$ such that $a\leq Cb$ . We write $a\approx b$ if $a\lesssim b\lesssim a$.


\subsection{Capacities and the capacity density condition}\label{secnew}

The set of (positive) Radon measures in $\R^{n+1}$ is denoted by $M_+(\R^{n+1})$. The Hausdorff $s$-dimensional measure and Hausdorff $s$-dimensional content are denoted by $\HH^s$ and $\HH^s_\infty$, respectively.

The fundamental solution of the negative Laplacian in $\R^2$ is
$$\EE_2(x) = \frac1{2\pi}\,\log \frac1{|x|},$$
while in higher dimensions, i.e., in $\R^{n+1}$, $n\geq2$, it equals
$$\EE_{n+1}(x) = \frac{c_n}{|x|^{n-1}},$$
where $c_n= (n-1)\HH^n(\mathbb S^n)$, with $\mathbb S^n$ being the unit hypersphere in $\R^{n+1}$.
In any case, for a measure $\mu$ in $\R^{n+1}$, we consider the energy
$$I(\mu) = \iint \EE_{n+1}(x-y)\,d\mu(x)\,d\mu(y)$$ 
and, for $F\subset\R^{n+1}$ we define the capacity
$$\capp(F) = \frac 1{\inf_{\mu\in M_1(F)} I(\mu)},
$$
where the infimum is taken over all {\em probability} measures $\mu$ supported on $F$. 
In the planar case, $\capp(F)$ is the Wiener capacity of $F$, and in higher dimensions this is the Newtonian capacity. 
From now on, in the plane we will write $\capp_W(F)$ instead of $\capp(F)$.
In fact, in the plane it is more convenient to work with the logarithmic capacity, defined by
$$\capp_L(F) = e^{-\frac{2\pi}{\capp_W(F)}}.$$

Recall that we denote by $\omega$ (and sometimes $\omega_\Omega$) the harmonic measure on an open set $\Omega$.  The following result is well known.
See for example Lemma 2.1 from \cite{Tolsa-IMRN}.

\begin{lemma}
\label{lembourgain}
 For $n\geq2$, let $\Omega\subset \R^{n+1}$ be open and let $B$ be a closed ball centered in $\partial\Omega$. Then 
\[ \omega^{x}(B)\geq c(n) \frac{\capp(\tfrac14 B\setminus\Omega)}{\rad(B)^{n-1}}\quad \mbox{  for all }x\in \tfrac14 B\cap \Omega ,\]
with $c(n)>0$.
\end{lemma}

An analogous result holds in the plane. The precise statement is the following. 
\begin{lemma}\label{lemCL1}
Let $\Omega\subset \R^{2}$ be open and let $B$ be a closed ball centered in $\partial\Omega$. Then 
$$ \omega^{x}(B)\gtrsim 
 \frac1{\log\dfrac{\rad(B)}{\capp_L(\frac14B\setminus\Omega)}}\qquad \mbox{ for all $x\in \tfrac14 B\cap \Omega$.}
 $$
\end{lemma}

I provide the detailed proof below because it is not easy to find in the literature.

\begin{proof}
Denote $r=\rad(B)$.
Replacing $\Omega$ by $\frac1{4r}\,\Omega$ if necessary, we can assume that $\diam(B)<1$.
Then, denoting $F=\frac14B\setminus\Omega$, the following identity  holds:
$$\capp_W(F)=\sup\big\{\mu(F):\mu\in M_+(\R^2),\,\supp\mu\subset F,\|\EE_2*\mu\|_{\infty}\leq1\big\}.$$

Let $\mu$ be the optimal measure for this supremum, so that $\supp\mu\subset F$, $\mu(F) =\capp_W(F)$, and the function $u:= \EE_2*\mu$ is harmonic out of $F$ and it satisfies 
$\|u\|_\infty\leq 1$. For all $z\in\frac14 B$ and all $y\in F$ we have $|z-y|\leq \frac12\,r$. Therefore,
$$u(z) = \frac1{2\pi} \int \log\frac1{|z-y|}\,d\mu(y) \geq \frac1{2\pi} \int \log\frac2{r}\,d\mu(y) = \frac{\mu(F)}{2\pi} \,\log\frac2{r}\qquad \mbox{ for all $z\in \tfrac14 B$.}
 $$
Also, for $z\in B^c$, we have $\dist(z,\supp F)\geq \frac34r(B)$, and thus
$$u(z) \leq \frac1{2\pi}\int \log\frac4{3r} \,d\mu(y) = 
\frac{\mu(F)}{2\pi} \,\log\frac4{3r}\qquad \mbox{ for all $z\in B^c$.}
 $$
 
Consider now the function
$$v = u - \frac{\mu(F)}{2\pi} \,\log\frac4{3r}.$$
Observe that
$$v(z) \geq \frac{\mu(F)}{2\pi} \,\log\frac2{r} - \frac{\mu(F)}{2\pi} \,\log\frac4{3r}
= \frac{\mu(F)}{2\pi} \,\log\frac32 \qquad \mbox{ for all $z\in \tfrac14 B$}
 $$
and
$$v(z) \leq 0\qquad \mbox{ for all $z\in B^c$.}
 $$
By the maximum principle and the fact that $x\in\frac14B$ we deduce that
$$\omega^x(B) \geq \frac{v(x)}{\sup v}\geq \frac{\mu(F)}{2\pi\sup v} \,\log\frac32 = c\,\frac{\capp_W(F)}{\sup v}.
$$
Regarding $\sup v$, taking into account that $\|u\|_\infty\leq1$, it is clear that
$$\sup v\leq 1- \frac1{2\pi}\,\log\frac4{3r}\,{\mu(F)}  = 1 - \frac1{2\pi}\,\log\frac4{3r}\,\capp_W(F)
\leq 1 - \frac1{2\pi}\,\log\frac1{r}\,\capp_W(F).$$
Therefore,
$$\omega^x(B) \geq c\,\frac{\capp_W(F)}{1 - \frac1{2\pi}\,\log\frac1{r}\,\capp_W(F)} =
c'\,\frac{1}{\log\dfrac{1}{\capp_L(F)} - \log\dfrac1{r}} = c'\,\frac{1}{\log\dfrac{r}{\capp_L(F)}}. 
$$
\end{proof}

\vv

\begin{lemma}\label{lemcap-cont}
Let $E\subset\R^{n+1}$ be compact and $n-1<s\leq n+1$. In the case $n>1$, we have
$$\capp(E) \gtrsim_{s,n} \HH_\infty^s(E)^{\frac{n-1}s}.$$
In the case $n=1$, we have
$$\capp_L(E) \gtrsim_s \HH_\infty^s(E)^{\frac1s}.$$
\end{lemma}

This result is an immediate consequence of Frostman's Lemma. See \cite[Chapter 8]{Mattila} for the case $n>1$, and 
\cite[Lemma 4]{Cufi-Tolsa-Verdera} for the case $n=1$, for example.

\vv

Let $\Omega\subsetneq\R^{n+1}$ be open, and let $\xi\in\partial\Omega$ and $r_0>0$. We say that $\Omega$ satisfies the $(\xi,r_0)$-local capacity density 
condition if there exists some constant $c>0$ such that, for any $r\in(0,r_0)$,
$$\capp(B(\xi,r)\setminus\Omega) \geq c\,r^{n-1}\qquad \mbox{in the case $n>1$},$$
and
$$\capp_L(B(\xi,r)\setminus\Omega) \geq c\,r\qquad \mbox{in the case $n=1$}.$$
We say that $\Omega$ satisfies the capacity density condition (CDC) if it satisfies the $(\xi,\diam(\Omega))$-local capacity density condition uniformly for all $\xi\in\partial\Omega$. In particular, if $\partial\Omega$ is $s$-AD regular for some $s>n-1$, then $\Omega$ satisfies the CDC.

Remark that, by the previous lemmas, if $\Omega$ satisfies the CDC, for any ball $B$ centered in $\partial \Omega$, it holds
$$\omega^{x}(B)\gtrsim 1 \quad \mbox{ for all $x\in \tfrac14 B\cap \Omega$}.$$
The following result is also standard and well known. See \cite{Ancona}, for example.

\begin{lemma}\label{lemholder1}
Let $\Omega\subset  \R^{n+1}$, let $\xi\in \partial\Omega$, and let $r>0$. Suppose that
$\Omega$ satisfies the $(\xi,r_0)$-local capacity density condition. 
 Let $u$ be a nonnegative function which is continuous in
$\overline{B(\xi, r)\cap \Omega}$ and harmonic in $B(\xi, r)\cap \Omega$, and vanishes continuously on $B(\xi,r)\cap\partial\Omega$. Then there is $\alpha>0$ such that for all $r\in (0,r_0)$,
\begin{equation}
u(x)\lesssim  \left(\frac{|x-\xi|}{r}\right)^{\alpha}\sup_{B(\xi,r)\cap \Omega} u \;\;\quad\mbox{ for all } x\in \Omega\cap B(\xi,r).
\label{e:holder}
 \end{equation}
 Both $\alpha$ and the implicit constant above depend only on $n$ and the constant $c$ involved in the definition of the
  $(\xi,r_0)$-local capacity density condition.
\end{lemma}

\vv

\subsection{Corkscrews, connectivity conditions, and uniform domains}
All domains $\Omega$ considered in this paper are Wiener regular, that is, the Dirichlet problem for the Laplacian with data in $C(\partial\Omega)$ is solvable in $\Omega$. Remark that the CDC condition for $\Omega$ or the local CDC condition at each $\xi\in\partial\Omega$ imply the Wiener regularity of $\Omega$.

A domain $\Omega\subset\R^{n+1}$ is called uniform (or $C$-uniform) if for every $x,y\in\Omega$ there is a curve $\gamma\subset \R^{n+1}$ 
connecting $x$ and $y$ such that
\begin{itemize}
\item[(a)] $\HH^1(\gamma)\leq C\,|x-y|$, and
\item[(b)] for all $z\in\gamma$, $\dist(z,\partial\Omega)\geq C^{-1}\,\dist(z,\{x,y\})$.
\end{itemize}

We say that  $\Omega$ satisfies the corkscrew condition (or $c$-corkscrew condition) if there exists some constant $c > 0$ such that for all $\xi\in\Omega$ and all $0<r \leq\diam(\partial\Omega)$ there is a ball $B(x,cr)\subset B(\xi,r)\cap\Omega$. The point $x$ is called a ``corkscrew point'' relative to the ball $B(\xi,r)$.
Further, we say that $x$ is a ``John point'' relative to $B(\xi,r)$ if for every $y\in B(\xi,r)\cap\Omega$
there exists a curve $\gamma$ satisfying the properties (a), (b) above.
It is immediate to check that if $\Omega$ is a uniform domain, then it satisfies the corkscrew condition. Also, for these domains, any corkscrew point is a John point relative to the ball associated to the corkscrew.

We will need to use the following result below. This is proven exactly in the same way as the analogous result for NTA domains in
\cite[Lemma 4.4]{Jerison-Kenig}.

\begin{lemma}\label{lem***}
For $n\geq 1$, let $\Omega\subset\R^{n+1}$ be a domain and $\xi\in\partial\Omega$, $0<r\leq \diam(\partial \Omega)$ such that
the $(\xi,r)$-local CDC holds. Let $x\in\Omega\cap B(\xi,r)$ be a corkscrew John point relative to $B(\xi,r)$. 
Let $u$ be a nonnegative function which is continuous in
$\overline{B(\xi, r)\cap \Omega}$ and harmonic in $B(\xi, r)\cap \Omega$, and vanishes continuously on $B(\xi,r)\cap\partial\Omega$. Then 
we have
$$\sup_{B(\xi,r/2)\cap \Omega} u \lesssim u(x),$$
with the implicit constant depending on the local CDC, and the corkscrew and John properties of $x$.
\end{lemma}

In \cite[Lemma 2.2]{DFM} the following result has been proved.

\begin{lemma}
For $0<s<n$ and $C_0>1$, let $E$ be an $(s,C_0)$-AD regular set. Then the domain $\Omega=\R^{n+1}\setminus E$ is $C$-uniform, with $C$ depending
just on $n,s,C_0$.
\end{lemma}

So the assumptions in Theorems \ref{teo1} or \ref{teo2} ensure that the domain $\Omega$ in those theorems is uniform and satisfies the CDC.
This is very useful because it implies some nice properties for the associated harmonic measure. For example, it implies that
given $p\in\R^{n+1}$ with $\dist(p,E)\gtrsim\diam(E)$, the doubling property $\omega^p(2B)\lesssim \omega^p(B)$ holds for any
ball $B$ centered in $\partial\Omega$. This is proven in \cite{Jerison-Kenig}
for NTA domains (i.e., for uniform domains $\Omega$ such that $\R^{n+1}\setminus \overline \Omega$ is a corkscrew domain), but the same proof works for uniform domains satisfying the CDC.

The following theorem states the so called ``change of pole formula" for uniform domains. This is proven in \cite{Jerison-Kenig}
for NTA domains, while the more general statement below is from \cite{MT}.

\begin{theorem}\label{teounif}
For $n\geq 1$, let $\Omega\subset\R^{n+1}$ be a Wiener regular uniform domain and let $B$ be a ball centered at $\partial\Omega$.
Let $p_1,p_2\in\Omega$ such that $\dist(p_i,B\cap \partial\Omega)\geq c_0^{-1}\,r(B)$ for $i=1,2$.
Then, for any Borel set $E\subset B\cap\partial\Omega$,
$$\frac{\omega^{p_1}(E)}{\omega^{p_1}(B)}\approx \frac{\omega^{p_2}(E)}{\omega^{p_2}(B)},$$
with the implicit constant depending only on $c_0$ and the uniform behavior of $\Omega$. 
\end{theorem}

\vv

\subsection{The Kelvin transform}

Given $x\in\R^{n+1}\setminus\{0\}$, we let 
$x^*= \frac1{|x|^2}\,x.$ 
Notice that the map defined by $T(x)=x^*$ is an involution of $\R^{n+1}\cup\{\infty\}$, understanding that
$T(0)=\infty$.
For $r>0$, we denote $r^*= \frac{1}{r}$.
Given an unbounded open set $\Omega\subset\R^{n+1}$ such that $\partial \Omega$ is compact and $0\not\in \partial\Omega$, we set
$$\Omega^* =
\big\{x^*:x\in\Omega\big\}\cup \{0\}$$ 
(so identifying $\Omega$ with $\Omega\cup\{\infty\}$, we have $\Omega^*=T(\Omega)$).
Given a function $u:\R^{n+1}\supset \overline\Omega\to \R$, vanishing at $\infty$ in the case $n>1$, its Kelvin transform is defined
by
$$ u^*(x^*)= \frac{1}{|x^*|^{n-1}}\,u(x),$$
understanding that $u(\infty)=0$ in the case $n>1$.
It is well known that $\Delta u=0$ in $\Omega$ if and only if $\Delta (u^*)=0$ in $\Omega^*$. Further, $(u^*)^*=u$.  See \cite[Chapter 4]{ABR}, for example. 
 Hence, if we take $u(x) = \omega_\Omega^x(F)$ for $F\subset \partial\Omega$, it follows that $u^*$ is a harmonic function in $\Omega^*$, vanishing in $\partial\Omega^*\setminus F^*$, and comparable to $1$ in $F^*$, with the implicit constant depending on $\dist(0,\partial\Omega)$
 and $\diam(\partial\Omega)$. Hence, for $x\in\Omega$, we have
 \begin{equation}\label{eqkelvin}
 \omega_{\Omega^*}^{x^*}(F^*)
  \approx u^*(x^*) = \frac1{|x^*|^{n-1}}\,u(x) \approx \omega_\Omega^x(F),
\end{equation}
with the first implicit constant depending on $\dist(0,\partial\Omega),\diam(\partial\Omega)$, and the second one on $|x^*|.$


\vv

\subsection{The Main Lemma}

To prove Theorems \ref{teo1} and \ref{teo2} we will use the following result, first implicitly used by
 Bourgain in \cite{Bourgain}, later by Batakis \cite{Batakis}, and more recently by Azzam \cite{Azzam-drop}.
 
\begin{lemma}\label{lem-dim} 
For $n\geq 1$, $s>0$, $C_0>1$, there exists an $M=M(n,s,C_0)>1$ (sufficiently big)
 such that the following holds. Let $E\subset \R^{n+1}$ be an $(s,C_0)$-AD regular set.
Let $\nu$ be a measure supported on
$E$ and $c_1\in (0,1)$ such that, for all $x\in E$, $0<r\leq \diam(E)$, there exists a ball $B(y,\rho)$ with $y\in B(x,r)\cap E$, $c_1\,r\leq \rho\leq r$, satisfying either
\begin{equation}\label{eq1--}
\frac{\nu(B(y,\rho))}{\rho^s} \geq M \,\frac{\nu(B(x,r))}{r^s}\qquad \mbox{or} \qquad \frac{\nu(B(y,\rho))}{\rho^s} \leq M^{-1} \,\frac{\nu(B(x,r))}{r^s}.
\end{equation}
Then $\dim \nu <s$.
\end{lemma} 
 
Although this result is not stated as above in \cite{Azzam-drop}, the detailed arguments of the proof are contained in Section 5 of that paper.

Given $\Omega$ and $E$ as in Theorems \ref{teo1} and \ref{teo2}, 
we denote by $p_0$ the pole for harmonic measure and we assume that this is far away from $E$ (in the plane we could take $p_0=\infty$), and we denote $\omega=\omega^{p_0}$. Recall that $\Omega=E^c$ is a uniform domain satisfying the CDC.
Thanks to Lemma \ref{lem-dim}, the fact that $\dim \omega < s$ is an immediate consequence of the following result.

\begin{mlemma}\label{mainlemma}
For $s>0$, $C_0>1$, let $E\subset \R^{n+1}$ be an $(s,C_0)$-AD regular closed set contained in a hyperplane. Let $M\geq 1$ and suppose either that $s\in[n-\frac12,n)$ or $s\in(n-\frac12-\ve,n-\frac12)$ for some $\ve= \ve(n,s,C_0,M)>0$ small enough. For any $x\in E$, $0<r\leq \diam(E)$, and $p_0\in \R^{n+1}\setminus (E\cup B(x,2r))$, there exists a ball $B(y,\rho)$ with $y\in B(x,r)\cap E$, $c\,r\leq \rho\leq r$, with
$c>0$ depending just on $n,s,C_0,M$, satisfying either
\begin{equation}\label{eq1}
\frac{\omega^{p_0}(B(y,\rho))}{\rho^s} \geq M \,\frac{\omega^{p_0}(B(x,r))}{r^s}\qquad \mbox{or} \qquad \frac{\omega^{p_0}(B(y,\rho))}{\rho^s} \leq M^{-1} \,\frac{\omega^{p_0}(B(x,r))}{r^s}.
\end{equation}
\end{mlemma}

To prove the Main Lemma, first we will show that, given $E$, $B(x,r)$, and $p_0$ as above, there exists $y\in B(x,r)\cap E$ such that
\begin{equation}\label{eq2}
\frac{\omega^{p_0}(B(y,\rho))}{\omega^{p_0}(B(x,r))} \geq c(s,C_0)\, \left(\frac{\rho}{r}\right)^{n-\frac12}
\end{equation}
for all $\rho\in (0,c'\, r)$, for some $c'\in(0,1)$ depending only on $s$, $C_0$, and $n$.
Clearly, this yields \rf{eq1} for $s\in (n-\frac12,n)$ and $\rho$ small enough. For the cases $s=n-\frac12$ and $s\in(n-\frac12-\ve,n-\frac12)$ we will need more careful estimates.
Notice also that the estimate \rf{eq2} is independent of the pole $p_0$, modulo a constant factor (as soon as $p_0\in \R^{n+1}\setminus (E\cup B(x,2r))$), because $\R^{n+1}\setminus E$ is a uniform domain. 

Sections \ref{sec3} and \ref{sec4} below are devoted to the proof of Main Lemma \ref{mainlemma}.
\vv

\section{The arguments for the planar case $n=1$ with $s\in [\frac12,1)$}\label{sec3}

\subsection{Proof of \rf{eq2}}\label{sec2.1}
Let $E\subset\R^2$ be as in the Main Lemma \ref{mainlemma}.
Without loss of generality, we assume that $E\subset\R\equiv \R\times\{0\}$. 
Let $x\in E$ and $0<r\leq \diam E$.
Taking into account that $s<1$, by a pigeon-hole argument, there is
an open interval $I=(a,b)\subset [x-r,x]$ which does not intersect $E$ and satisfies $\ell:=\HH^1(I)\approx_s r$. By enlarging $I$ if necessary, we can assume that $b\in E$. Notice that $b$ is contained in $[x-(1-c)r,x]$ because $x\in E$, for some $c>0$ depending on $s$.

We choose $y=b$. Again by the $s$-AD regularity of $E$ and the pigeon-hole principle, there exist radii $r_1,r_2$ with 
$\ell/2\leq r_1< r_2 \leq \ell$, $r_2-r_1\approx_s \ell\approx r$ such that
$$A(y,r_1,r_2)\cap E=\varnothing.$$
Here $A(y,r_1,r_2)$ stands for the open annulus centered in $x$ with inner radius $r_1$ and outer radius $r_2$.
Observe that the left component of $A(y,r_1,r_2)\cap\R$ is contained in $I$.

Next we apply a ``localization argument". We denote $E_1 = E\cap \bar B(y,r_1)$, $\Omega_1= E_1^c$, $r'=(r_1+r_2)/2$.
It is immediate to check that $E_1$ is still $s$-AD regular and thus $\Omega_1$ is a uniform domain too.
We claim that for any subset $F\subset E_1$ and any $p\in \partial B(y,r')$,
\begin{equation}\label{eq3}
\omega_{1}^p(F)\approx_s \omega^p(F),
\end{equation}
where $\omega_1$ stands for the harmonic measure for $\Omega_1$. 
To prove the claim, consider first $p\in \partial B(y,r')$ such that
$$\omega_1^p(F) = \max_{q\in \partial B(y,r')}\omega_1^q(F).$$
Using that $\omega^z_1(F)$ is harmonic in $\Omega$ and vanishes in $E_1\setminus F$ and the maximum principle, we get
\begin{align*}
\omega_1^p(F) &= \int_E \omega_1^z(F)\,d\omega^p(z) = \omega^p(F) +  \int_{E\setminus E_1} \omega_1^z(F)\,d\omega^p(z)\leq \omega^p(F) + \sup_{z\in E\setminus E_1}\omega_1^z(F)\,\,\omega^p(E\setminus E_1).
\end{align*}
Observe that, by Lemma \ref{lembourgain}, Lemma \ref{lemcap-cont},
the CDC, and a Harnack chain argument,
$\omega^p(E_1)\geq \delta_0,$
for some $\delta_0>0$ depending just on $s$. 
Hence, $\omega^p(E\setminus E_1)\leq 1-\delta_0$. 
Also, since $\omega_1^z(F)$ is harmonic in $\C_\infty\setminus B(y,r')$ and $E\setminus E_1\subset\C_\infty\setminus B(y,r')$, by the maximum principle we have
$$\sup_{z\in E\setminus E_1}\omega_1^z(F) \leq \max_{q\in \partial B(y,r')}\omega_1^q(F) = \omega_1^p(F).$$
Therefore,
$$\omega_1^p(F)\leq \omega^p(F) + \omega_1^p(F)\,(1-\delta_0),$$
or equivalently,
$\omega_1^p(F) \leq \delta_0^{-1}\,\omega^p(F).$
By the definition of $r'$ and Harnack's inequality, we infer
$$\omega_1^p(F) \lesssim \omega^p(F)$$
for {\em all} $p\in \partial B(y,r')$. On the other hand, by the maximum principle, we have trivially that
$\omega_1^p(F) \geq \omega^p(F)$, which concludes the proof of the claimed estimate \rf{eq3}.

Next we will perform another modification of the domain $\Omega_1$. 
For a fixed $\rho\in (0,r_1/4)$, consider the intervals $J=[y,y+\rho/2]$, $J'=[y,y+\rho]$ and define 
$E_2=E_1\cup J$ and
$\Omega_2=E_2^c =\Omega_1\setminus J$. 
By the CDC and the uniformity of $\Omega_1$, we infer that, for all $q\in \partial B(y,\rho/2)$,
$$\omega_1^q(J'\cap E_1)\gtrsim 1\geq \omega_2^q(J).$$
We also have $\omega_1^q(J'\cap E_1)\geq\omega_2^q(J)=0$ for all $q\in J^c\cap E_1$. Then, by the maximum principle, since both $\omega_1^z(J'\cap E_1)$ and $\omega_2^z(J)$ are harmonic in $\Omega_1\setminus \bar B(y,\rho/2)=\Omega_2\setminus \bar B(y,\rho/2)$
we deduce that
$$\omega_1^q(J'\cap E_1)\gtrsim \omega_2^q(J)$$
for all $q\in \Omega_2\setminus \bar B(y,\rho/2)$, and in particular for all $p\in \partial B(y,r')$. 

Finally we let $E_3 = [y,y+r_1]$ and $\Omega_3=E_3^c$, so that $E_2\subset E_3$. By the maximum principle, we have
$$\omega_2^p(J)\geq \omega_3^p(J)$$
for all $p\in \partial B(y,r')$. Hence, gathering the above estimates, we infer that, for all $p\in \partial B(y,r')$,
\begin{equation}\label{eq5}
\omega^p(J'\cap E)\approx_s \omega_1^p(J'\cap E)\gtrsim \omega_2^p(J) \geq \omega_3^p(J).
\end{equation}

Now it just remains to estimate $\omega_3^p(J)$. We can do this by means of a conformal transformation. Indeed, observe first that, by a
Harnack chain argument and the maximum principle, $\omega_3^p(J)\approx \omega_3^\infty(J)$ for all $p\in \partial B(y,r')$.
Next, suppose for simplicity that $y=-r_1/2$, so that $E_3= [-r_1/2,r_1/2]$.
The map $f:\bar B(0,1) \to \overline{\Omega_3}$ defined by
\begin{equation}\label{eq6}
f(z) = \Big(z + \frac1z\Big)\frac{r_1}4
\end{equation}
is a conformal transformation from $B(0,1)$ to $\Omega_3$ such that $f(0)=\infty$, with $f(\partial B(0,1)) = \partial\Omega_3=E_3$.
Thus,
$$\omega_3^\infty(J)= \frac1{2\pi}\,\HH^1(f^{-1}(J)).$$
An easy computation shows that
$$f^{-1}(J)= \{e^{i\alpha}: \pi-\theta\leq \alpha\leq\pi+\theta\},$$
with 
\begin{equation}\label{eqtheta}
\theta= \arccos \bigg(1-\frac{2\HH^1(J)}{r_1}\bigg) = \arccos \Big(1-\frac\rho{r_1}\Big)\approx \Big(\frac\rho{r_1}\Big)^{1/2}.
\end{equation}
Thus,
\begin{equation}\label{eqtheta2}
\omega_3^\infty(J)= \frac \theta{\pi} \approx \Big(\frac\rho{r_1}\Big)^{1/2}\approx \Big(\frac\rho{r}\Big)^{1/2}.
\end{equation}
Consequently, by the change of pole formula for uniform CDC domains and \rf{eq5}, we deduce that, for $p\in \partial B(y,r')$, 
\begin{equation}\label{eqcanvipol}
\frac{\omega^{p_0}(B(y,\rho))}{\omega^{p_0}(B(x,r))} \approx \omega^{p}(\bar B(y,\rho)) 
 = \omega^p(J'\cap E) \gtrsim \omega_3^p(J) \approx \Big(\frac\rho{r}\Big)^{1/2},
 \end{equation}
which completes the proof of \rf{eq2}.

\vv

\subsection{The case $s=1/2$} \label{sec2.2}
In this case the inequality \rf{eq2} does not suffice to prove \rf{eq1} and we need a better estimate.
We consider the preceding domains $\Omega_1,\Omega_2,\Omega_3$, so that,  for all 
$p\in \partial B(y,r')$,  \rf{eq5} holds. However, the estimate $\omega_2^p(J) \geq \omega_3^p(J)$ is too coarse for our purposes.
Instead, we write
$$\omega_2^p(J) \approx \omega_2^\infty(J)= \int_{E_3} \omega_2^z(J)\,d\omega_3^\infty(z).$$
The density $\frac{d\omega_3^\infty}{d\HH^1|_{E_3}}$ can be computed explicitly by means of the conformal transformation in \rf{eq6}.
Using the identity $\omega_3^\infty(J)=\pi^{-1}\arccos \left(1-\frac{2\HH^1(J)}{r_1}\right)$ and differentiating,
it follows that
$$ \frac{d\omega_3^\infty}{d\HH^1|_{E_3}}(t) = \frac1{\pi\sqrt{(\frac{r_1}2-t)(\frac{r_1}2+t})}.$$
Thus,
\begin{equation}\label{eq07}
\frac{d\omega_3^\infty}{d\HH^1|_{E_3}}(t)\approx \frac1{\sqrt{r_1(t+\frac{r_1}2})}\quad \mbox{ for $t\in[-r_1/2,0]$,}
\end{equation}
and so
\begin{equation}\label{eq7}
\omega_2^p(J) \gtrsim  \int_{\frac{-r_1}2}^0 \omega_2^t(J)\,\frac{dt}{\sqrt{r_1(t+\frac{r_1}2})}
\end{equation}
(recall that we are identifying $\R\equiv \R\times\{0\})$.

To estimate the integral in \rf{eq7} from below, consider the annuli $A_k= A(y,2^k\rho,2^{k+1}\rho)$ for $k\geq 1$, and let $N = [\log_2 \frac {r_1}\rho]$.
By the $s$-AD regularity of $E$ and pigeonholing, for every $k\in[1,N]$ there exists an interval $I_k\subset A_k\cap E_3$ (recall $E_3$ is an interval) such that $\HH^1(I_k)\approx 2^k\rho$ and $I_k\cap E= I_k\cap E_2=\varnothing$. Let $\hat I_k$ be another interval concentric with $I_k$ and half length. Then we write
\begin{equation}\label{eq7.5}
\int_{\frac{-r_1}2}^0 \omega_2^t(J)\,\frac{dt}{\sqrt{r_1(t+\frac{r_1}2})} \geq \sum_{k=1}^N 
\int_{\hat I_k} \omega_2^t(J)\,\frac{dt}{\sqrt{r_1(t+\frac{r_1}2})}.
\end{equation}

We claim that 
\begin{equation}\label{eq8}
\omega_2^t(J)\gtrsim \bigg(\frac\rho{|t-y|}\bigg)^{1/2}\quad \mbox{ for all $t\in \bigcup_{k=1}^N \hat I_k$.} 
\end{equation}
Assuming this for the moment, we obtain
\begin{align*}
\omega_2^p(J) & \gtrsim \sum_{k=1}^N 
\int_{\hat I_k} \bigg(\frac\rho{|t-y|}\bigg)^{1/2}\,\frac{dt}{\sqrt{r_1(t+\frac{r_1}2})} \\
& \approx
\sum_{k=1}^N 
\int_{\hat I_k} \bigg(\frac\rho{\HH^1(\hat I_k)}\bigg)^{1/2}\,\frac{dt}{r_1^{1/2}\HH^1(\hat I_k)^{1/2}} = N\,\Big(\frac\rho{r_1}\Big)^{1/2} \approx \log\frac r\rho\,\Big(\frac\rho{r}\Big)^{1/2}.
\end{align*}
By \rf{eq5} and the change of pole formula, arguing as in the preceding subsection, we obtain
\begin{equation}\label{eq11}
\frac{\omega^{p_0}( B(y,\rho))}{\omega^{p_0}(B(x,r))} \approx \omega^{p}(\bar B(y,\rho)) 
 = \omega^p(J'\cap E) \gtrsim \omega_2^p(J) \approx \log\frac r\rho\,\Big(\frac\rho{r}\Big)^{1/2},
 \end{equation}
which implies \rf{eq1} for $\rho$ small enough.

It remains to prove \rf{eq8}. To this end, for each $t\in\hat I_k$,
let $t'\in \R$ be the point symmetric to $t$ with respect to $y$. That is, $t'=-r_1 - t$. Notice that $t'$ is on the left side of the interval $E_3$ (recall that the leftmost point of $E_3$ is $y=-r_1/2$). By a Harnack chain argument and the maximum principle, we have
$$\omega^t_2(J)\approx \omega_2^{t'}(J) \geq \omega_3^{t'}(J).$$
Now we can compute explicitly $\omega_3^{t'}(J)$ by means of the conformal transformation
in \rf{eq6}. Indeed, consider the change of variable $t'=y - \frac{r_1}2\,h$. Then, it follows easily that
$$f^{-1}(t') = \frac{-1}{1+h+\sqrt{h(2+h)}} = -(1+h) +\sqrt{h(2+h)}.$$
So $f^{-1}(t')$ is a point in the unit disk belonging to the segment $(-1,0)$ such that 
$$|-1-f^{-1}(t')| = -h +\sqrt{h(2+h)} \approx h^{1/2} = \bigg(\frac{2|t'-y|}{r_1}\bigg)^{1/2}
.$$
Recall that $f^{-1}(J)= [\pi-\theta,\pi+\theta]$, with $\theta\approx\big(\frac\rho{r_1}\big)^{1/2}$, by \rf{eqtheta}. Hence,
$|-1-f^{-1}(t')|\gtrsim \HH^1(f^{-1}(J))$.
Taking into account that, for any point $q\in B(0,1)$ and $\eta := 10\,\dist(q,\partial B(0,1))$, $\omega_{B(0,1)}^q|_{B(q,\eta)}$ is comparable to 
$\eta^{-1}\HH^1|_{\partial B(0,1)\cap B(q,\eta)}$, we deduce that
$$\omega_3^{t'}(J) = \omega_{B(0,1)}^{f^{-1}(t')}(f^{-1}(J)) \approx \frac{\theta}{|-1-f^{-1}(t')|}\approx 
 \frac{\big(\frac\rho{r_1}\big)^{1/2}}{\Big(\frac{2|t'-y|}{r_1}\Big)^{1/2}} \approx  \frac{\rho^{1/2}}{|t'-y|^{1/2}}= \frac{\rho^{1/2}}{|t-y|^{1/2}}, 
$$
which yields \rf{eq8}.
\fiproof

\vvv


\section{The higher dimensional case $n>1$}\label{sec4}

In this section we complete the proof of Main Lemma \ref{mainlemma}. Although we will focus mainly in the case $n>1$, the arguments
in this section are also valid for the case $n=1$. In the preceding section we have studied separately the case $n=1$, 
$s\in [1/2,1)$ because the arguments are more elementary, especially for $s=1/2$.

\subsection{An auxiliary lemma}

In the case $n>1$, we cannot use conformal transformations as in the planar case. Instead, we will use the following auxiliary
result. 

\begin{lemma}\label{lemaux}
For $n>1$, $r>0$, let $ E= (B(0,2r) \setminus B(0,r)) \cap (\R^n\times \{0\})$ and $\Omega = \R^{n+1}\setminus E$.
Let $B$ be a ball centered in $E$ such that $\rho := \rad(B) \approx \dist(B,\partial B(0,r))$, with $0<\rho\leq r/2$. Then we have
$$\omega^0(B) \gtrsim \Big(\frac \rho r\Big)^{n-\frac12},$$
where $\omega^0$ stands for the harmonic measure for $\Omega$ with pole in $0$.
\end{lemma}

Quite likely this result can be proven by an explicit computation of the density $\frac{d\omega^0}{d\HH^n|_E}$. However, 
I have not found in the literature such computation, which seems to be non-trivial. For this reason, below I show a 
different argument relying on the planar case.

\begin{proof}
Observe first that, by rescaling, we can assume that $r=1$, and we can estimate $\omega^{p_0}(B)$ for $p_0=(0,\ldots,0,1/2)$, say,
instead of $\omega^0(B)$, since $\omega^0(B)\approx \omega^{p_0}(B)$.
Next, by applying the Kelvin transform $T x= \frac x{|x|^2}$, the statement in the lemma is equivalent to saying that, for
$\Omega_1= \R^{n+1}\setminus E_1$ with $E_1 = T(E) = (B(0,1) \setminus B(0,1/2)) \cap (\R^n\times \{0\})$, and any ball $B$ centered 
in $E_1$ such that $\rho := \rad(B) \approx \dist(B,\partial B(0,1))$, with $0<\rho\leq 1/2$, it holds
$$\omega_1^{p_1}(B) \gtrsim  \rho^{n-\frac12},$$
where $\omega_1^{p_1}$ stands for the harmonic measure for $\Omega_1$ with pole in $p_1:=T(p_0)=(0,\ldots,0,2)$.
Consider now $\Omega_2= \R^{n+1}\setminus E_2$ with $E_2 = B(0,1)  \cap (\R^n\times \{0\})$. By the maximum principle, it is clear 
that
$\omega_1^{p_1}(B) \geq \omega_2^{p_1}(B).$
Hence it suffices to prove that 
\begin{equation}\label{eqrk5*}
\omega_2^{p_1}(B) \gtrsim  \rho^{n-\frac12}.
\end{equation}

To prove \rf{eqrk5*}, we may assume that the ball $B$ is centered in the segment $L\subset\R^{n+1}$ connecting the origin to the point $(1,0,\ldots,0)$. Then
we let $Q$ be an $n$-dimensional cube concentric with $B$ (with sides parallel to the axes), contained in $E_2\cap B$, with side length comparable to $\rad(B)$, and we set $R=[-1,1]^n\times \{0\}$. Since $Q\subset B\cap E_2$ and
$E_2\subset R$, by the maximum principle we have
\begin{equation}\label{eqmpr5*}
\omega_2^{p_1}(B) \geq \omega_R^{p_1}(Q),
\end{equation}
where $\omega_R$ stands for the harmonic measure for the domain $\R^{n+1}\setminus R$.
To estimate $\omega_R^{p_1}(Q)$ from below, we can assume that this is an $n$-dimensional dyadic cube descendant of $R$, by translating and reducing $Q$ suitably if necessary.
 We let $I_0 = [-1,1]$ and $J=L\cap Q$, so that, abusing notation, $J\subset I_0$
and $\dist(J,\partial I_0)\approx \ell(J)$, where we are identifying $\R=\R\times \{0_{\R^n}\}\subset\R^{n+1}$.
Considering now $I_0,J$ as subsets of $\R\equiv \R\times \{0\}\subset \C$, we let
$$u(z) = \omega_{I_0}^z(J),$$
where $\omega_{I_0}^z$ is the harmonic measure for the domain $\C\setminus I_0$ with pole in $z$. Next we denote 
$\wt I_0= I_0   \times \R^{n-1}\times \{0\}$, $\wt\Omega_0 = \R^{n+1}\setminus \wt I_0$, $\wt J =  J   \times \R^{n-1}\times \{0\}$ and we take the
function $f:\wt \Omega_0\to\R$ defined by
$f(x) = u(x_1,x_{n+1})$, for $x=(x_1,\ldots,x_{n+1})$. Notice that $f$ is non-negative, bounded by $1$,  and harmonic in
$\R^{n+1}\setminus \wt I_0$. Moreover, it extends continuously to $0$ in $\wt I_0\setminus \overline{\wt J}$, and to $1$ in $J^\circ \times \R^{n-1} \times \{0\} $ (here $J^\circ$ is the interior of $J$ in $\R$, i.e., $J$ minus its end-points). Then it follows that $f(x) =\wt \omega^x(\wt J)$, where $\wt\omega^x$ is the harmonic measure for $\wt\Omega_0$ with pole in $x$.

We claim that 
\begin{equation}\label{eqclaim1*}
\wt \omega^{p_1}(\wt J)\approx \wt \omega^{p_1}(\wt J \cap R).
\end{equation}
To prove this, denote $g(x) =\wt \omega^{x}(\wt J \cap R)$, so that  
$$f(x) - g(x) = \wt \omega^x(\wt J \setminus R).$$
For a given $\delta\in(0,1/2)$, let $z_\delta\in\C\setminus I_0$ be such that $\dist(z_\delta,I_0)=\delta $ and
$$u(z_\delta) = \max_{z\in\C:\dist(z,I_0)=\delta} u(z) = \max_{z\in\C:\dist(z,I_0)=\delta} \omega_{I_0}^z(J),$$
so that by the maximum principle $u(z) \leq u(z_\delta)$ for any $z\in\C$ such that $\dist(z,I_0)\geq\delta$.
Let $q_\delta= (\Re z_\delta,0,\ldots,0,\Im z_\delta)$. 
By the definition of $f$, it turns out that 
\begin{equation}\label{eqfmax}
f(x) \leq f(q_\delta)\quad \mbox{ for all $x\in\R^{n+1}$ such that $\dist(x,\wt I_0)\geq \delta$.}
\end{equation}
Then, by the (local) CDC of $\wt I_0$, since $f-g$ vanishes identically in $R\times\{0\}$, we deduce that
$$|f(x)-g(x)| \lesssim 
\sup_{y\in 3I_0\times (\frac34 I_0)^n}
|f(y)-g(y)|\, \dist(x,R)^\alpha\quad \mbox{ for all $x\in 2I_0\times (\frac12 I_0)^n$.}$$
Since $p_1$ is a local John point for $\wt\Omega_0$ in $4I_0\times I_0^n$,
the supremum above is bounded by $C|f(p_1)-g(p_1)|$, by Lemma \ref{lem***}. Then, choosing $x=q_\delta$ 
and taking into account that $f(p_1) \leq f(q_\delta)$ by \rf{eqfmax},
we obtain
$$|f(q_\delta)-g(q_\delta)| \lesssim |f(p_1)-g(p_1)|\, \dist(q_\delta,R)^\alpha\leq f(p_1)\, \delta^\alpha \leq f(q_\delta)\, \delta^\alpha.$$
So, choosing $\delta$ small enough (independent of $f$ and $g$), we derive $f(q_\delta)\approx g(q_\delta)$.
Since $f(p_1) \approx f(q_\delta)$ and $g(p_1)\approx g(q_\delta)$ (with implicit constants depending on $\delta$), we deduce \rf{eqclaim1*}, as claimed.

Next we split
$R= \bigcup_{i=1}^{N} P_i$, where $P_i$'s are dyadic $n$-dimensional cubes, descendants of $R$, with the same side length as $Q$,
and $N=\ell(Q)^{-n}$. Observe that $\wt J\cap R$ coincides with the union of $N_J:=\ell(Q)^{1-n}$ of the cubes $P_i$. So reordering the
family of such cubes if necessary, we have
$$\wt \omega^{p_1}(\wt J \cap R) = \sum_{i=1}^{N_J}\wt \omega^{p_1}(P_i).$$
It is easy to check that $\wt \omega^{p_1}(P_i)\approx \wt \omega^{p_1}(P_j)$ for all $1\leq i,j\leq N_J$. Indeed, if $c_i$ stands for the center of $P_i$ and we take the corkscrew point $z_i = c_i+(0,\ldots,0,1)$, by symmetry we have
$\wt \omega^{z_i}(P_i)= \wt \omega^{z_j}(P_j)$ for all $i,j$, and by Harnack, $\wt \omega^{z_i}(P_i)\approx \wt \omega^{p_1}(P_i)$ for all $i$.
Recalling that $Q$ coincides with one of the cubes $P_i$, using also \rf{eqclaim1*}, we get
$$\wt \omega^{p_1}(Q) \gtrsim \frac1{N_J}\sum_{i=1}^{N_J}\wt \omega^{p_1}(P_i) = \ell(Q)^{n-1}\,\wt \omega^{p_1}(\wt J \cap R)
\approx \ell(Q)^{n-1}\,
\wt \omega^{p_1}(\wt J).
$$
By construction, we have $\wt \omega^{p_1}(\wt J)= \omega_{I_0}^{(0,1)}(J)$, and using the conformal transformation in
\rf{eq6} and the identity \rf{eq07}, it is immediate to check that $\omega_{I_0}^{(0,1)}(J)\approx \ell(J)^{1/2}= \ell(Q)^{1/2}$. Therefore,
$$\wt \omega^{p_1}(Q) \gtrsim \ell(Q)^{n-\frac12}.$$
By the maximum principle, we have $\omega_R^{p_1}(Q)\geq \wt \omega^{p_1}(Q)$, and then by \rf{eqmpr5*}, we deduce
that 
$$\omega_2^{p_1}(B) \gtrsim \ell(Q)^{n-\frac12} \approx \rho^{n-\frac12},$$
which proves \rf{eqrk5*} and the lemma.
\end{proof}

\vv

\subsection{Proof of \rf{eq2}}\label{sec2.1*}

We assume that $E\subset\R^n\equiv \R^n\times\{0\}$. By a pigeon-hole argument, there is
an open ball $B_1:=B(x_1,r_1)\subset\R^{n+1}$ centered in $\R^n$, which satisfies:
$$2\bar B_1\subset B(x,r),\qquad
B_1\cap E = \varnothing,\qquad \partial B_1\cap E\neq\varnothing,\qquad r_1\approx r.$$
We choose $y\in \partial B_1\cap E$. 

As in the case $n=1$, we intend to apply 
 a localization argument. We denote $E_1 = E\cap \bar B(x_1,2r_1)$, $\Omega_1= E_1^c$.
We claim that for any subset $F\subset E_1\cap  B(x_1,1.5r_1)$,
\begin{equation}\label{eq3**}
\omega_{1}^{x_1}(F)\approx_s \omega^{x_1}(F),
\end{equation}
where $\omega_1$ stands for the harmonic measure for $\Omega_1$. 
To prove the claim, consider first $p\in \partial B(x_1,1.8r_1)$ such that
$$\omega_1^p(F) = \max_{q\in \partial B(x_1,1.8r_1)}\omega_1^q(F).$$
Using the fact that $\omega^z_1(F)$ is harmonic in $\Omega$ and vanishes in $E_1\setminus F$, we get
\begin{align*}
\omega_1^p(F) &= \int_E \omega_1^z(F)\,d\omega^p(z) = \omega^p(F) +  \int_{E\setminus E_1} \omega_1^z(F)\,d\omega^p(z)\leq \omega^p(F) + \sup_{z\in E\setminus E_1}\omega_1^z(F)\,\,\omega^p(E\setminus E_1).
\end{align*} 
Since $\omega_1^z(F)$ is harmonic in $\R^{n+1}\setminus E_1$, vanishes in $E_1\setminus \bar B(x_1,1.5r_1)$ and 
at $\infty$,  and $E\setminus E_1\subset\R^{n+1}\setminus \bar B(x_1,1.5r_1)$, by the maximum principle we have
$$\sup_{z\in E\setminus E_1}\omega_1^z(F) \leq \max_{q\in \partial B(x_1,1.8r_1)}\omega_1^q(F) = \omega_1^p(F).$$
By the H\"older continuity of $\omega_1^z(F)$ in $\partial B(x_1,1.8r_1)$, which is far away from $F$ and $E\setminus E_1$,
and by the uniformity of $\Omega$ and Lemma \ref{lem***},
we derive
$$\sup_{q\in \partial B(x_1,1.8r_1):\dist(q,\partial\Omega)\leq \delta}\omega_1^q(F) \lesssim \bigg(\frac{\delta}{r_1}\bigg)^\alpha\sup_{q\in A(x_1,1.6r_1, 2r_1)}\omega_1^q(F)\approx \bigg(\frac{\delta}{r_1}\bigg)^\alpha\sup_{\substack{q\in \partial B(x_1,1.8r_1):\\ \dist(q,\partial\Omega)\geq r_1}}\omega_1^q(F)
,$$
for some $\alpha>0$ depending on $s$. So we derive that  $\dist(p,E)\approx r_1$.

Since $p$ is a corkscrew point for $\Omega$ relative to $B(y,r_1)$, by the CDC and a Harnack chain argument,
$\omega^p(E_1)\geq \delta_0,$
for some $\delta_0>0$ depending just on $s$. 
Hence, $\omega^p(E\setminus E_1)\leq 1-\delta_0$. 
Therefore,
$$\omega_1^p(F)\leq \omega^p(F) + \omega_1^p(F)\,(1-\delta_0),$$
or equivalently,
$\omega_1^p(F) \leq \delta_0^{-1}\,\omega^p(F).$
Since both $x_1$ and $p$ are corkscrew points for $\Omega$ relative to $B(y,2r_1)$, we deduce that 
$$\omega_1^{x_1}(F) \lesssim \omega^{x_1}(F).$$
 On the other hand, by the maximum principle, we have trivially
$\omega_1^{x_1}(F) \geq \omega^{x_1}(F)$, which concludes the proof of the claimed estimate \rf{eq3**}.

Next we will perform another modification of the domain $\Omega_1$. For $0<\rho\leq r_1/4$, we denote $E_2 = E_1\cup \big(\bar B(y,\rho/2)\cap \R^n\setminus B_1\big)$ and we let $\Omega_2=\R^{n+1}\setminus E_2$.
By arguments analogous to the ones for the case $n=1$, we infer that
\begin{equation}\label{eq4**}
\omega_1^{x_1}(B(y,\rho))\gtrsim \omega_2^{x_1}(\bar B(y,\rho/2))\geq \omega_1^{x_1}(\bar B(y,\rho/2)).
\end{equation}
Finally, we let $E_3 = (2\bar B_1 \setminus B_1)\cap \R^n$ and $\Omega_3 = \R^{n+1}\setminus E_3$. By the maximum principle and Lemma~\ref{lemaux}, we have
$$\omega_2^{x_1}(\bar B(y,\rho/2))\geq \omega_3^{x_1}(B(y,\rho/2)) \gtrsim \Big(\frac\rho r\Big)^{n-\frac12}.
$$
Together with \rf{eq3**} and \rf{eq4**}, 
and the change of pole formula, as in \rf{eqcanvipol}, this yields \rf{eq2}.

\vv


\subsection{The case $s\in (n-\frac12-\ve,\,n-\frac12]$}\label{sec2.3*}

In this case, instead of obtaining an explicit estimate for 
$\frac{\omega(B(y,\rho))}{\omega(B(x,r))}$, we will argue by contradiction. 

We assume that $\ve\leq 1/4$. In this situation it is easy to check that the estimates obtained in Subsection \ref{sec2.1*} 
involving the domains $\Omega_1,\Omega_2,\Omega_3$
are uniform in $s$, assuming the AD-regularity constant to be also uniform in $s$.
In the rest of this subsection we assume also that all the implicit constants involved in the relations $\lesssim$, $\gtrsim$, $\approx$ are uniform in $s$, for $s\in (n-\frac12-\ve,\,n-\frac12]$.

For a given $M\gg1$, suppose that
\begin{equation}\label{eq10*}
M^{-1}\Big(\frac{\rho}r\Big)^s\leq \frac{\omega^{p_0}(B(y,\eta))}{\omega^{p_0}(B(x,r))}\leq M\Big(\frac{\rho}r\Big)^s
\quad \mbox{ for $\rho\leq \eta\leq r$}.
\end{equation}
For $x_1$ as above, we write
$$\omega_2^{x_1}(\bar B(y,\rho/2)) = \int_{E_3} \omega_2^z(\bar B(y,\rho/2))\,d\omega_3^{x_1}(z).$$
Consider the annuli $A_k= A(y,2^k\rho,2^{k+1}\rho)$, for $k\geq 1$. 
Let $N = [\log_2 \frac {r_1}\rho]$.
By the $s$-AD regularity of $E$ and pigeon-holing, for every $k\in[1,N]$ there exists a ball $B_k$ centered in $E_3$ such that
$2B_k\cap \R^n\subset A_k\cap E_3\setminus E$ with $\rad(B_k)\approx 2^k\rho$. So we have
$$\omega_2^{x_1}(\bar B(y,\rho/2)) \geq \sum_{k=1}^N \int_{E_3\cap B_k} \omega_2^z(\bar B(y,\rho/2))\,d\omega_3^{x_1}(z)
\gtrsim \sum_{k=1}^N\omega_3^{x_1}(B_k)\,\inf_{z\in B_k} \omega_2^z(\bar B(y,\rho/2)).$$
By Lemma \ref{lemaux}, we have
$$\omega_3^{x_1}(B_k) \gtrsim  \Big(\frac{2^k\rho} r\Big)^{n-\frac12}.$$
By the maximum principle and the change of pole formula, for any $z\in B_k$,
$$\omega_2^z(\bar B(y,\rho/2)) \geq \omega_1^z(\bar B(y,\rho/2)) \geq \omega^z(B(y,\rho/2)) \approx 
\frac{\omega^{p_0}(B(y,\rho/2))}{\omega^{p_0}(B(y,2^k\rho))}\geq M^{-2} \frac{\rho^s}{(2^k\rho)^s} = M^{-2}\,2^{-ks}.$$

In the case $s=n-\frac12$, we obtain
$$\omega_2^{x_1}(\bar B(y,\rho/2)) \gtrsim M^{-2} \sum_{k=1}^N \Big(\frac{2^k\rho} r\Big)^{n-\frac12}\,2^{-k(n-\frac{1}2)}=
M^{-2} N\,\Big(\frac{\rho} r\Big)^{n-\frac12}.$$
By the change of pole formula and the relationship between $\omega$, $\omega_1$, $\omega_2$, this yields
\begin{equation}\label{eq67*}
\frac{\omega^{p_0}(B(y,\rho))}{\omega^{p_0}(B(x,r))}\approx
\omega^{x_1}(B(y,\rho)) \approx \omega_1^{x_1}(B(y,\rho)) \gtrsim \omega_2^{x_1}(\bar B(y,\rho/2)) \gtrsim M^{-2} N\,\Big(\frac{\rho} r\Big)^{n-\frac12}.
\end{equation}
For $N$ big enough, this contradicts the assumption \rf{eq10*}. Hence \rf{eq1} holds, with $\rho$ in \rf{eq1}
replaced by some $\eta$ such that $\rho\leq \eta\leq r$.

For $s\in (n-\frac12-\ve,\,n-\frac12)$, we have
$$\omega_2^{x_1}(\bar B(y,\rho/2)) \gtrsim M^{-2} \sum_{k=1}^N \Big(\frac{2^k\rho} r\Big)^{n-\frac12}\,2^{-ks}=
M^{-2} \Big(\frac{\rho} r\Big)^{n-\frac12} \sum_{k=1}^N 2^{k(n-\frac12-s)}
.$$
Assuming $N$ big enough, or equivalently $\rho$ small enough (depending on $|n-\frac12-s|$),  we write
$$ \sum_{k=1}^N 2^{k(n-\frac12 -s)} \geq \sum_{k=0}^{N-1} 2^{k(n-\frac12 -s)} =  \frac{2^{N(n-\frac12-s)} - 1}{2^{n-\frac12 -s} - 1}\approx
\frac{2^{N(n-\frac12-s)}}{2^{n-\frac12 -s} - 1} \approx \frac{r^{n-\frac12 -s}}{(n-\frac12-s)\,\rho^{n-\frac12 -s}}.
$$
Hence,
$$\omega_2^{x_1}(\bar B(y,\rho/2))  \gtrsim M^{-2}\,\frac{\rho^s}{(n-\frac12-s)\,r^s}.$$ 
As in \rf{eq67*}, this implies
$$\frac{\omega^{p_0}(B(y,\rho))}{\omega^{p_0}(B(x,r))}\gtrsim M^{-2}\,\frac{\rho^s}{(n-\frac12-s)\,r^s}\geq M^{-2}\,\frac{\rho^s}{\ve\,r^s}.$$
For $N$ big enough and $\ve$ small enough, this contradicts \rf{eq10*} and so \rf{eq1} holds, with $\rho$ in \rf{eq1}
replaced by some $\eta$ such that $\rho\leq \eta\leq r$.
\fiproof


\section{Extension to subsets of $C^1$ manifolds}

In this section we extend the main results of the paper  
to $s$-AD regular subsets of $n$-dimensional $C^1$ manifolds. We prove the following.

\begin{theorem}\label{teoC1}
For $n\geq1$ and $s\in [n-\frac12,n)$, let $E\subset \R^{n+1}$ be an $(s,C_0)$-AD regular compact set contained in an $n$-dimensional $C^1$ manifold. Let $\Omega=\R^{n+1}\setminus E$ and denote
by $\omega$ the harmonic measure for $\Omega$. Then
$\dim \omega < s.$
\end{theorem}

Given $E\subset\R^{n+1}$, $x\in\R^{n+1}$ and $r>0$, we denote
$$\beta_{\infty,E} (x,r) = \inf_L \sup_{y\in B(x,r)\cap E}\frac{\dist(y,L)}r,$$
where the infimum is taken over all the hyperplanes $L\subset\R^{n+1}$.

To prove Theorem  \ref{teoC1}, we will use the following variant of the Main Lemma \ref{mainlemma}.

\begin{lemma}\label{lemcompactness}
For $n\geq1$, let $s\in[n-\frac12,n)$ and $C_0>0$. 
For all $M\geq 1$ there are constants $\gamma>0,\,c>0$ both depending on $n,s,C_0,M$ such that the following holds. Let $E\subset\R^{n+1}$ be an $(s,C_0)$-AD regular compact set, and $x\in E$, $0<r\leq \diam(E)$ such that
 $\beta_{\infty,E}(x,\gamma^{-1}r)\leq \gamma^2$, and $p\in \R^{n+1}\setminus (E\cup B(x,2r))$. Then there exists a ball $B(y,\rho)$ with $y\in B(x,r)$, $c\,r\leq \rho\leq r$, such that either
\begin{equation}\label{eq1**}
\frac{\omega^p(B(y,\rho))}{\rho^s} \geq M \,\frac{\omega^p(B(x,r))}{r^s}\qquad \mbox{or} \qquad \frac{\omega^p(B(y,\rho))}{\rho^s} \leq M^{-1} \,\frac{\omega^p(B(x,r))}{r^s}.
\end{equation}
\end{lemma}

\vv
To prove this lemma, we will use a compactness argument involving the so-called Attouch-Wets topology. This is a local variant of the Hausdorff metric topology for convergence of sets. For $r>0$ and non-empty sets $E,F\subset\R^{n+1}$, we denote
$$d_r(E,F) = \max\Big( \sup_{x\in E\cap {\bar B(0,r)}}\dist(x,F),\,\sup_{x\in F\cap {\bar B(0,r)}}\dist(x,E)\Big).$$
We say that a sequence of sets $E_k\subset\R^{n+1}$
converges to a set $F\subset\R^{n+1}$ in the Attouch-Wets topology if, for every $r>0$, $d_r(E_k,F)\to 0$ as $k\to\infty$.
By arguments very similar to the ones used to prove the compactness of a sequence of compact subsets of a compact metric space in the Hausdorff distance topology, it follows that if $E_k$ is a sequence of closed subsets of $\R^{n+1}$ intersecting a fixed ball, then there is a subsequence $E_{k_j}$ that converges in the Attouch-Wets topology to some closed set in $\R^{n+1}$.
See \cite{BL} for a recent exposition on this topic.

\begin{proof}[Proof of Lemma \ref{lemcompactness}]
Suppose the lemma fails. That is, there exist $n>1$, $s\in[n-\frac12,n)$, $C_0>0$, and $M>1$ such that for all $\gamma=1/k$ and $c=1/k$
there exists an $(s,C_0)$-AD regular set $E_k\subset \R^{n+1}$  
and $x_k\in E_k$, $0<r_k\leq \diam(E_k)$, such that
 $\beta_{\infty,E_k}(x_k,kr_k)\leq k^{-2}$, and $p_k\in \R^{n+1}\setminus (E_k\cup B(x,2r_k))$, and moreover
for any ball  $B(y,\rho)$ with $y\in B(x_k,r_k)$, $\frac1k\,r_k\leq \rho\leq r_k$, it holds
\begin{equation}\label{eq1***}
M^{-1}\frac{\omega_k^{p_k}(B(x_k,r_k))}{r_k^s} < \frac{\omega_k^{p_k}(B(y,\rho))}{\rho^s} < M \,\frac{\omega_k^{p_k}(B(x_k,r_k))}{r_k^s},
\end{equation}
where $\omega_k$ is the harmonic measure for $\R^{n+1}\setminus E_k$.
By dilating and translating suitably $E_k$, we can assume that $x_k=0$ and $r_k=1$. Also, by the uniformity and the $s$-AD-regularity of the domain 
$\R^{n+1}\setminus E_k$, we can also assume that $p_k\in A(0,2,3)$ and $\dist(p_k,E_k) \geq c'>0$.
By considering a subsequence, we can assume that the sets $E_k$ converge in the Attouch-Wets topology to some $(s,C_0)$-AD regular set $\wt E\subset\R^{n+1}$, and also that $p_k$ converges to some point 
$\wt p\in \overline{ A(0,2,3)}$ and $\dist(\wt p,\wt E) \geq c'>0$. It easily follows then that
$\omega_k^{p_k}$ converges weakly * to $\wt\omega^{\wt p}$, the harmonic measure for $\wt\Omega=\R^{n+1}\setminus\wt E$ with pole in $\wt p$.

Observe that, for $0<R<k$, we have
$$\beta_{\infty,E_k}(0,R)\leq \frac kR\,\beta_{\infty,E_k}(0,k)\leq \frac1{kR}\to 0 \quad \mbox{ as $k\to\infty$.}$$
Hence we infer that $\beta_{\infty,\wt E}(0,R) = 0$ for every $R>0$. That is to say, $\wt E$ is contained in some hyperplane, which passes through the origin because $0\in\wt E$. On the other hand, from \rf{eq1***} and the weak * convergence of $\omega_k^{p_k}$ to $\wt\omega^{\wt p}$ we deduce that, for $x=0$, $r=1$,
$$M^{-1}\frac{\wt \omega^{\wt p}(B(x,r))}{r^s} \leq \frac{\wt \omega^{\wt p}(B(y,\rho))}{\rho^s} \leq M \,\frac{\wt \omega^{\wt p}(B(x,r))}{r^s}\quad \mbox{ for all $y\in B(0,r)$, $0<\rho\leq r$.}$$
This contradicts Main Lemma \ref{mainlemma}.
\end{proof}
\vv

\begin{proof}[Proof of Theorem \ref{teoC1}] Just notice that if $E$ is compact, $s$-AD regular with $s\in[n-\frac12,n)$, and it is contained in an $n$-dimensional $C^1$ manifold, then the assumptions of Lemma \ref{lemcompactness} hold for any ball $B(x,r)$ centered in $E$ with 
small enough radius $r$. It is easy to check that then the assumptions of Lemma \ref{lem-dim} hold (by adjusting suitably the constant $c_1$ in the lemma if necessary), which ensures that $\dim\omega<s$.
\end{proof}
\vv

Finally remark that, by similar compactness arguments, one can also extend Theorem \ref{teo2} to subsets of $n$-dimensional $C^1$ manifolds:

\begin{theorem}\label{teoC2}
 Let $E\subset \R^{n+1}$ be a compact set contained in an $n$-dimensional $C^1$ manifold which is $(s,C_0)$-AD regular for some $s>0$. 
 That is,
 $$C_0^{-1}r^s\leq\HH^s(E\cap B(x,r))\leq C_0\,r^s\quad \mbox{ for all $x\in E$ and $0<r\leq \diam(E)$.}$$
 Then there exists $\ve>0$ small enough depending on $C_0$ such that if $s\in (n-\frac12-\ve,n)$, we have $\dim \omega<s$.
 \end{theorem}

To prove this, it suffices to consider the case $s\in (n-\frac12-\ve,n-\frac12)$ and then one can use arguments similar to ones in Lemma \ref{lemcompactness}. The details are left for the reader.



\vv

\vv

\end{document}